\documentclass[12pt]{article}
\usepackage{titlepic}
\usepackage{pifont}
\usepackage{fancyhdr}
\usepackage{amsmath}
\usepackage{amsthm}
\usepackage{amssymb,amsfonts}
\usepackage{threeparttable}
\usepackage[mathscr]{eucal}
\usepackage{hyperref}
\usepackage[dvips]{graphicx}
\usepackage{graphics}
\usepackage{geometry}
\usepackage{times}
\usepackage{enumerate}
\usepackage{courier}
\usepackage[most]{tcolorbox}
\usepackage{soul,color}
\definecolor{paleaqua}{rgb}{0.74, 0.83, 0.9}
\definecolor{cadetblue}{rgb}{0.37, 0.62, 0.63}
\definecolor{bittersweet}{rgb}{1.0, 0.44, 0.37}
\definecolor{amethyst}{rgb}{0.6, 0.4, 0.8}
\definecolor{VioletRed4}{rgb}{ 0.55, 0.13, 0.32}
\definecolor{MediumPurple}{RGB}{147, 112, 219}
\definecolor{mypink3}{cmyk}{0, 0.7808, 0.4429, 0.1412}
\definecolor{mygray}{gray}{0.6}
\definecolor{airforceblue}{rgb}{0.36, 0.54, 0.66}
\usepackage{hyperref}
\hypersetup{colorlinks=true, %colorise les liens
breaklinks=true, %permet le retour `a la ligne dans les liens trop longs
urlcolor= blue, %couleur des hyperliens
linkcolor= VioletRed4, %couleur des liens internes
citecolor=VioletRed4, %couleur des r´ef´erences
anchorcolor = blue, citecolor = blue, filecolor = MidnightBlue!90
}
\newtheorem{theorem}{Theorem}[section]

\newtheorem{prop}[theorem]{Proposition}
\newtheorem{cor}[theorem]{Corollary}

\newtheorem{definition}[theorem]{Definition}
\newtheorem{example}[theorem]{Example}

\newtheorem{rem}{Remark}[section]

\numberwithin{equation}{section}

\providecommand{\keywords}[1] {
  \small
  \textbf{\textit{Keywords---}} #1
}

 \usepackage{authblk}

\title{Formal proof of the approximation of a random phenomenon by a chaotic phenomenon}

\author[1]{Mohamed El Ouafi}
\author[2]{Hajar Ahalli}
\author[,1]{Abderrahim Aslimani\thanks{Corresponding author: a.slimani@ump.ac.ma}}
\author[1]{Kaoutar Lamrini Uahabi}

\affil[1]{Department of
Mathematics, Faculty of Applied Sciences, Mohammed First University, Laboratory of Applied Mathematics and Information Systems (LMASI), Mathematics and Applications Team, BP
300-62700, Selouane, Nador, Morocco\\ \vspace{.4cm}}

\affil[2]{Department of Mathematics, Faculty of Sciences, Mohammed First University, Stochastic and Deterministic Modeling Laboratory, BP 717-60000,  Oujda, Morocco.}
\begin{document}
\maketitle
\begin{abstract}
This article examines the subtle relationship between chaos and randomness, two concepts that, although they refer to seemingly unpredictable phenomenon, are based on fundamentally different principles. Chaos manifests in deterministic systems where small variations in initial conditions lead to unpredictable long-term behaviors, while randomness pertains to intrinsically probabilistic processes, characterized by fundamental uncertainty.
Although these phenomena are based on distinct mechanisms, they can interact and converge in contexts as varied as the modeling of natural phenomena, climate forecasts, financial markets, etc.. Despite their differences, these two phenomena share common characteristics, such as the absence of apparent order and an unpredictability that defies our attempts at long-term prediction.
Through an analysis of chaos theory and probability, this article aims to clarify the distinctions and highlight the deep connections between these two concepts in real systems. The objective of this article is to present a comprehensive approach aimed at demonstrating that, under certain hypothesizes, a random phenomenon can be effectively represented and approximated by a chaotic phenomenon. By examining this possibility, we seek to establish the theoretical foundations that connect these two concepts, often perceived as distinct, but whose dynamics could prove to be analogous in certain contexts. \\

\end{abstract}

\vspace{.2cm}


\section{Introduction}
\mbox{~}

The modeling of random phenomena fundamentally relies on probability laws approximated by algorithmically generated pseudo-random sequences. This methodology, widely used in various fields including numerical simulation, modeling of natural phenomena, financial market analysis, artificial intelligence, and information system encryption, is rooted in deterministic chaotic systems. In artificial intelligence, stochastic algorithms underpin fundamental techniques such as deep learning optimization, Monte Carlo tree search, and generative models, while in cryptography, the security of encryption protocols depends critically on the quality of random number generation for cryptographic keys and initialization vectors. However, despite their well-documented empirical effectiveness, no formal proof has yet established the ability of these generators to faithfully reproduce random laws within a rigorous mathematical framework.

The emergence of this paradigm finds its roots in a remarkable conceptual evolution throughout the 20th century. Beginning in the 1970s, work in ergodic theory, particularly Bernoulli transformations and Ornstein's isomorphism (1972) \cite{Co}, began to mathematically formalize the appearance of randomness in deterministic systems. This theoretical foundation was enriched by Takahashi (1984) \cite{Tak}, who proposed an explicit conceptual bridge between chaos and statistical mechanics, establishing formal analogies between Lyapunov exponents and "internal energy," or between Kolmogorov entropy and thermodynamic entropy. Eckmann and Ruelle (1985) \cite{Ec} subsequently consolidated this framework by showing how concepts of unpredictability and fluctuations emerge naturally in nonlinear deterministic systems.

The development of this research trajectory continued with Ornstein (1989) \cite{Or}, who compared deterministic and random systems, demonstrating that statistical properties typically associated with randomness can originate from systems governed by deterministic laws. This perspective was complemented by the practical approach of James (1995) \cite{Ja}, examining how pseudo-random generators can exploit chaotic properties to produce sequences sufficiently random for demanding computational applications. Finally, Gallavotti and Cohen (1995) \cite{Ga} extended this connection to non-equilibrium systems through their "chaotic hypothesis," solidifying the idea that randomness observed in complex systems emerges intrinsically from the deterministic dynamics itself, rather than as mere external perturbation.
%%%%%%%%%%%%%%%%%%%%%%%%%%%%%%%%

Contemporary research attests to the enduring vitality of chaos-based approaches, with concrete applications across various technological domains. In this context, our recent work \cite{AH} advances the field by proposing a novel image encryption framework that integrates the stochastic Vasicek process with a modified chaotic logistic map. This hybrid approach leverages the stochastic properties of the Vasicek process and the sensitivity of chaotic dynamics to achieve enhanced security and robustness in image encryption. Following the Fridrich confusion-diffusion structure, their method generates process parameters and an initialization state from a secret key of arbitrary length. The Vasicek process performs multiple iterations—the number of which is dynamically determined by the secret key—to produce floating-point values that are subsequently discretized. These discretized values serve a dual purpose: they generate the permutation for pixel confusion and produce pseudo-random sequences for pixel diffusion. The architecture supports multiple encryption rounds to enhance security. Experimental validation confirms the algorithm's robustness, including a large key space, resistance to statistical and differential attacks, and overall efficiency, demonstrating its suitability for secure image transmission.

Recent advances in memristive chaotic systems have been particularly fruitful. Deng et \emph{al.} (2025) \cite{De2} introduced a discrete memristive conservative chaotic map (DMCM), demonstrating enhanced ergodicity and randomness compared to dissipative systems, with successful FPGA implementation and superior bit error rate performance in secure communication applications. Similarly, research on bursting firings in memristive Hopfield neural networks has shown promising results for image encryption and hardware implementation \cite{Yu}. The work of Emin $\&$ Yaz (2024) \cite{Em} on the digital implementation of chaotic systems using the Nvidia Jetson AGX Orin platform demonstrates the practical application of these concepts for real-time signal processing. Similarly, chaos-based encryption techniques developed by Kıran (2024) \cite{Ki}, leveraging GPU parallelism, illustrate the potential of these systems for information security. Fundamental analyses of dynamical systems, such as the study of the discrete T-system by Rana (2023) \cite{Ra}, continue to provide the theoretical foundations necessary for these applied developments. Finally, chaotic oscillator designs on FPGA proposed by Taşdemir et \emph{al.} (2025) \cite{Ta} open promising perspectives for dedicated hardware implementation of high-performance random generators.\\

%%%%%%%%%%%%%%%%%%%%%%%%%%%%%%%%%%%%%%%%%ù

Our work aims to develop a mathematical framework demonstrating how random phenomena can be approximated by chaotic dynamics on metric compact sets. Specifically, we establish that any elementary event in a probability space can be approximated by a convergent subsequence of chaotic phenomena. This advancement transcends conventional testing of pseudo-random number generators which typically assess uniformity, independence, periodicity, and autocorrelation by providing a rigorous mathematical foundation for the approximation of random phenomena by chaotic systems.

In what follows, we define and analyze discrete chaotic dynamical systems, introducing a new characterization of chaos that generalizes the classical Devaney framework \cite{De}. We first present the fundamental concepts of randomness and chaos, then develop the mathematical apparatus necessary to establish our main results concerning the formal approximation of random phenomena by deterministic chaotic systems.

%%%%%%%%%%%%%%%%%%%%%%%%%%%

%%%%%%%%%%%%%%%%%%%%%%%%%%%%%%%

\section{Chaos and randomness: approaches and convergences}

Let $\Omega$ be  an infinite locally compact topological space and $\mathcal{F}:=\mathcal{B}(\Omega)$ the associated Borel $\sigma$-algebra.

\begin{definition}

\begin{itemize}
    \item 
A (discrete) dynamical system is given by a pair $(\Omega, T)$ where $T : \Omega \to \Omega$ is a (often continuous) function representing the evolution of the dynamical system. 
\item Let $(\Omega, T)$ be a dynamical system and $\nu$ a nonzero Radon measure on $\Omega$. The dynamical system $(\Omega, T)$ is said to be (topologically) \textbf{chaotic} with respect to $\nu$ if for almost every $x_0 \in \Omega$ with respect to $\nu$, for every open set $ U \subset \Omega $, there exists an integer $ n \in \mathbb{N} $ such that $ T^n(x_0) \in U $. This means that the trajectories of $ T $ are dense in $ \Omega $ $\nu$-almost everywhere. We then say that the sequence $ (u_n) $ defined by

$$
\left\{
\begin{array}{ll}

u_{n+1}&=T\left(u_n\right) \\

u_0&=x_0,

 \end{array}
\right.
$$ 
 
is \textbf{chaotic}.
\end{itemize}

\end{definition}

Thus, a dynamical system $(\Omega, T)$ is chaotic if the iterations of points in $ \Omega $ can reach any open set in $ \Omega $ from $\nu$-almost any initial point. The application $ T $ is often called the evolution function or transition function of the dynamic system, and it describes how the state of the system evolves over time.

In what follows, the space $ \Omega $ is equipped with a distance $ d $, and when reference is made to the measure $ \nu $, it is understood to be a Radon measure on the measurable space $ (\Omega, \mathcal{F}) $. Moreover, the transformation $ T $ is assumed to be continuous on $ \Omega $.\\

We first state our main result, which stipulates that, there exists a probability measure $\mu$ on $(\Omega,\mathcal{F})$ such that the points of $\Omega$ (i.e., the elementary events) can be approximated by subsequences generated by a chaotic dynamic. In other words, within this framework, a random phenomenon can be approximated by a chaotic phenomenon.

\begin{theorem}
Suppose that the space $\Omega$ is compact. Then, any random phenomenon defined on $\Omega$ with  a sigma-algebra  $\mathcal{G}\subset\mathcal{F}$ (particularly when  $\mathcal{G}=\mathcal{F}$), can be approximated by a chaotic phenomenon on $\Omega$.
\end{theorem}

As consequence of this theorem we obtain:

\begin{cor}
    Suppose the dynamical system $(\Omega, T)$ is chaotic for some transformation $T: \Omega \to \Omega$ with $\Omega$ is compact. Then there exists a probability measure on $(\Omega, \mathcal{F})$ such that every element (which constitutes an elementary event) of $\Omega$ can be approximated by a convergent subsequence of a chaotic sequence of $(\Omega, T)$.
\end{cor}

To prove this theorem, we will need certain fundamental notions and preliminary results from the theory of dynamical systems and ergodicity.\\

\begin{definition} 
A dynamical system $ (\Omega,T) $ is said to be \textbf{sensitive to initial conditions} if for almost every $ x_0 \in \Omega $ with respect to $ \nu $, its trajectories $ (T^n(x_0))_{n\in\mathbb{N}} $ are sensitive to initial conditions. In other words, there exists $ \delta > 0 $ such that for every $ \varepsilon > 0 $, one can find a point $ y_0 \in \Omega $ such that $ d(x_0,y_0) < \varepsilon $ (initially close) and an integer $ n $ such that
$$
d(T^n(x_0), T^n(y_0)) > \delta.
$$
  
\end{definition}

This means that small initial disturbances can lead to large divergences after a certain period of time.

\begin{rem}
    \rm{In a dynamical system $(\Omega,T)$, if $T$ is continuous, the proximity of trajectories is preserved at each iteration. For example, if two points $x_0$ and $y_0$ in $\Omega$ are
very close, then $T(x_0)$ and $T(y_0)$ will be so as well, and so on. This means that a small perturbation in the initial state $x_0$ will lead to a small perturbation in the evolution of the system. So if continuity is respected, a small variation in the initial state will regularly affect the evolution of the system. However, this disturbance, although small at the beginning, eventually becomes significant over a
long period. This phenomenon of "preserved proximity" is crucial in the study of sensitivity to conditions
initials. Without continuity of $T$, two very close points can diverge abruptly at each iteration, which complicates the analysis of the dynamics.}
\end{rem}

\begin{prop}
If the dynamical system $(\Omega,T)$ is chaotic with respect to $\nu$, then it is sensitive to initial conditions. 
\end{prop}
 
\begin{proof}
Since the dynamical system $(\Omega,T)$ is chaotic with respect to $\nu$, we can take $ x_0 \in \Omega $ such that its trajectories $ \left(T^n(x_0)\right) $ are dense in $ \Omega $. Let $ U $ be a small compact neighborhood of $ x_0 $ and consider a perturbed fixed point $ y_0 \in U $. We assume, by contradiction, that for every $ \delta>0 $, the distance between $ T^n( x_0) $ and $ T^n (y_0) $ always remains less than or equal to $ \delta>0 $ for all $ n $. To do this, we first show that there exists $ n_0 \in \mathbb{N}^{\ast} $ such that 
$$
d(T^{n_0}(x_0), T^{n_0}(y_0)) > 0 .
$$ 
Indeed, if for all $ n \in \mathbb{N}^{\ast} $ we have $ T^{n}(x_0) = T^{n}(y_0) $ and this for all $ y_0 \in U $, then the application $ T^{n} $ will be constant on $ U $ for all $ n \in \mathbb{N}^{\ast} $ and thus $ \overline{(T^n(x_0))} \subset U $, which is impossible because the sequence $ (T^n(x_0)) $
is dense in $\Omega$. Therefore there exists $\delta_0\in\mathbb{R}$ such that 
$$
d(T^{n_0}(x_0),T^{n_0}(y_0))>\delta_0>0,
$$
which is absurd with the hypothesis. Consequently, the dynamical system $(\Omega,T)$ is sensitive to initial conditions.
\end{proof}

\begin{rem}
    \rm{Sensitivity to initial conditions is a fundamental property of chaotic systems. In a finite space, this property is impossible to satisfy. Indeed, let $(\Omega, T)$ be a dynamical system where $\Omega$ is a \textbf{finite} space equipped with a distance $d$. Let us denote by $d_{\text{min}}$ the smallest distance between two distinct points
$$
d_{\text{min}}: = \min\{d(x,y) \mid x, y \in \Omega, x \neq y\}.
$$
Let $x_0\in\Omega$ be fixed. To contradict sensitivity to initial conditions it suffices to choose $\varepsilon = \frac{d_{\text{min}}}{2}>0$, and we necessarily have $y_0 = x_0$ for all $y_0\in B(x_0,\varepsilon)$. Thus, for all $n \geq 0$, $T^n(x_0) = T^n(y_0)$ and therefore 
$$ d(T^n(x_0), T^n(y_0)) = 0, \ \forall n\geq 0.
$$
}
\end{rem}

\begin{definition}

A dynamical system $(\Omega, T)$ is said to be topologically \textbf{transitive} if for all non-empty open sets $ U, V $ in $ \Omega $, there exists $ n \geq 0 $ such that
$$
T^{-n}(U) \cap V \neq \emptyset.
$$

\end{definition}

This means that, if the dynamical system $(\Omega,T)$  is transitive, any open set $U$ can be transported by the iterations of the transformation $T$ sufficiently close to any other open set $V$ after some rank $n_0$. 

\begin{prop}
If the dynamical system $(\Omega,T)$ is chaotic with respect to $\nu$, then it is topologically transitive.
\end{prop}

\begin{proof}
 Let consider two non-empty open sets $ U $ and $ V $ in $ \Omega $. Since the trajectories are $\nu$-almost everywhere dense, for $\nu$-almost every point $ v \in V $, there exists an $ n_0 \in \mathbb{N} $ such that $ T^{n_0}(v) \in U $. So there exists a rank $ n_0 $ such that $ T^{-n_0}(U) \cap V \neq \emptyset $. This shows that the dynamical system $(\Omega,T)$ is transitive.
\end{proof}

\subsection{Invariant Measure}

\begin{definition}
Let $(\Omega,T)$ be a dynamical system. An \textbf{invariant measure} $ \mu $ for $(\Omega, T)$ is a probability measure on $( \Omega, \mathcal{F} )$ such that
$$
\mu(T^{-1}(A)) = \mu(A),
$$ 
for all $ A \in \mathcal{F} $.
\end{definition}

The definition means that the probability of occurrence of an event $ A $ does not change under the application $T$, that is, the distribution of points remains the same under the dynamic evolution $T$.

\begin{example}

\rm{
 The Gauss-Kuzmin-Wirsing transformation (often simply called the Gauss transformation) is a transformation on the set $\Omega=[0, 1]$ defined by 
$$
T_G(x): =\left\{
  \begin{array}{ll}
     \frac{1}{x} - \lfloor \frac{1}{x} \rfloor, & \text{if } x\in(0,1] \\
    0 & \text{if } x=0
  \end{array}
  \right.
$$
where $\lfloor \cdot \rfloor$ denotes the integer part function.  This transformation is closely related to continued fractions and has interesting dynamical properties \cite{Wi}. Let's examine its chaotic behavior. For a real number $x \in (0, 1]$, we can write its continued fraction as 
$$
x = \frac{1}{a_1 + \frac{1}{a_2 + \frac{1}{a_3 + \dots}}},
$$
where $a_1, a_2, \dots$ are positive integers. The transformation $T_G$ has an invariant measure $\mu$ given by
$$
\mu_G(A) = \int_A \frac{1}{\ln 2} \cdot \frac{1}{1 + x} \, dx, \ \forall  A \in \mathcal{F}.
$$
To show that this measure is $T_G$-invariant, we use the following result: if $ T : \Omega \subset \mathbb{R}^d \to \Omega $ is a local $ \mathcal{C}^1 $-diffeomorphism, and $ \rho $ is a continuous function, then the measure $\mu$ with density $\rho$ with respect to the Lebesgue measure is $ T $-invariant if and only if it satisfies the Frobenius-Perron functional equation
%\[
%\sum_{x \in f^{-1}(y)} \frac{\rho(x)}{| \det Df(x) |} = \rho(y)
%\]
$$
\sum_{x \in f^{-1}(y)} \frac{\rho(x)}{|T'(x)|} = \rho(y).
$$
For this, let's start by observing that each $ y $ has exactly one preimage $x_k$ in each interval $ \left( \frac{1}{1+p}, \frac{1}{p} \right] $, given by
$$
T(x_p) = \frac{1}{x_p} - p = y \quad \Leftrightarrow \quad x_p = \frac{1}{y + p}
$$
Moreover, notice that $ T'_G(x) = -\frac{1}{x^2} $ on each interval $ \left( \frac{1}{1+p}, \frac{1}{p} \right] $. 
\begin{align*}
 \sum_{x \in T_G^{-1}(y)} \frac{\rho_G(x)}{|T'_G(x)|} &=\sum_{p=1}^{\infty} \frac{1}{\ln 2} \cdot \frac{ x_p^2}{1 + x_p}\\
 &=\frac{1}{\ln 2} \sum_{p=1}^{\infty} \frac{1}{(y + p)(y + p + 1)} \\
 &= \rho_G(y).
\end{align*}
This shows that the measure $\mu$ with density $\rho_G$ is invariant under the Gauss transformation.}

\end{example} 

\begin{example}

\rm{Let $\Omega=[0,1]$. The logistic map is defined by the transformation $T_L$ on $\Omega$ as follows
$$
T_L(x) = \lambda x (1 - x),
$$
where $ \lambda $ is a real parameter in the interval $[0,4]$ and $ x \in \Omega $. This transformation acts on a point $ x_n $ to produce the next in the sequence 
$$
x_{n+1} = T_L(x_n) = \lambda x_n (1 - x_n).
$$
This iteration generates an orbit $ (x_n) $, which depends on the initial condition $ x_0 $ and the properties of $ T_L $. For $ \lambda= 4 $, the transformation $ T_L $ is associated with an absolutely continuous invariant measure $\mu_L$ with respect to the Lebesgue measure, whose invariant density is given by 
$$
\rho_L(x) = \frac{1}{\pi \sqrt{x(1 - x)}}\mathrm{1}_{(0,1)}(x), \ \forall  x\in \mathbb{R}.
$$
Again to show that the measure $\mu_L$ is invariant, we prove tha $\rho_L$ qatisfies the Frobenius-Perron functional equation. The preimages under the logistic transformation $T_L$, we solve the equation $ T_L(x) = y $, i.e., 
$$
4x(1 - x) = y.
$$
The solutions  of this equation are
$$
x_1 = \frac{1 + \sqrt{1 - y}}{2} \quad \text{and} \quad x_2 = \frac{1 - \sqrt{1 - y}}{2}.
$$
Thus, the preimage of $ y $ under $ T_L $ is $ T_L^{-1}(y) = \{x_1, x_2\}$. Knowing that $\rho_L(x_1)=\rho_L(x_2)$ and
$$
|T'_L(x_1)| = |T'_L(x_2)| = 4\sqrt{1 - y},
$$
we obtain 
\begin{align*}
 \sum_{x \in T_L^{-1}(y)} \frac{\rho_L(x)}{|T'_L(x)|} &=\frac{\rho_L(x_1) }{2\sqrt{1 - y}}= \rho_L(y).
\end{align*} 
Thus, $\rho_L$ satisfies the Frobenius-Perron functional equation.
  }  
\end{example}

\begin{definition}
Let $(\Omega,T)$ be a dynamical system and $\mu$ a probability measure on $\Omega$. We say that the measure $ T $ is \textbf{$ \mu $-ergodic} if, for every measurable set $ A \in \mathcal{F}$, invariant under $T$, that is, $ T^{-1}(A) = A $, we have
$$
\mu(A) = 0 \quad \text{or} \quad \mu(A) = 1.
$$
\end{definition}

In other words, there are no invariant subsets under $T$ that have a measure $\mu$ strictly between 0 and 1.\\

Now, we are ready to prove Theorem 2.2,

\begin{proof} (of Theorem 2.2.) We proceed to demonstrate this theorem in 4 steps:
\begin{itemize}
\item \emph{Step 1.} Since $\Omega$ is metric and compact, it is separable, meaning there exists a dense sequence $(x_n)$ in $\Omega$. Let $\Omega_0 = \{x_n : n \in \mathbb{N} \}$ and define the map 
\begin{align*}
    T_0 :& \ \Omega_0 \to \Omega_0 \subset \Omega \\
        & \ x_n \mapsto x_{n+1}
\end{align*}
for all $n \geq 0$. It is clear that $T_0$ is continuous on  $\Omega_0$ by construction. Since $\Omega$ is compact and metric, and $ T_0$ is defined on a dense set, by the extension theorem, there exists a continuous map $T : \Omega \to \Omega$ such that $T = T_0$ on $\Omega_0$.
Thus, we have constructed a continuous map $T : \Omega \to \Omega $ and a point $x_0 \in \Omega$ such that the sequence $\left( T^n(x_0) \right)_{n \geq 0}$ is dense in $\Omega$.

\item \emph{Step 2.} We now construct a probability measure $\mu$ on  $(\Omega, \mathcal{F})$ that adequately describes the distribution of the elements of $\Omega$. Consider a nonzero Radon measure $\nu$ on $\Omega$.
For $A \in \mathcal{F}$, let the measure $\mu_0$ be defined by 
$$
\mu_0(A) = \frac{1}{\nu(\Omega)} \nu(A).
$$
Next, define the sequence of measures  $(\mu_n)$ by
$$
\mu_n(A) = \mu_{n-1}(T^{-1}(A)),
$$
for all $A \in \mathcal{F}$. It is clear that the sequence $(\mu_n)$ consists of probability measures on $\Omega$. Since $T$ is a continuous transformation and $\Omega$ is compact, by Prokhorov's theorem \cite{Pr}, there exists a subsequence of measures $(\mu_{k_n})$ of$(\mu_n)$ that converges weakly to a limit measure $\mu$, i.e., for any continuous function $f$ on $\Omega$, we have 
$$
\lim_{n \to \infty} \int_\Omega f \, d\mu_{k_n} = \int_\Omega f \, d\mu.
$$
In particular, $\mu(\Omega) = 1$.

We now show that the limit measure $\mu$ is invariant under $T$. Indeed, for any continuous function $f$ on $\Omega$, we have
\begin{align*}
    \int_\Omega f \, d\mu &= \lim_{n \to \infty} \int_\Omega f \, d\mu_{k_n} \\
    &= \lim_{n \to \infty} \int_\Omega f \circ T \, d\mu_{k_{n-1}} \\
    &= \int_\Omega f \circ T \, d\mu.
\end{align*}
Thus, for any measurable set $A \in \mathcal{F}$, we have
$$
\mu(T^{-1}(A)) = \mu(A),
$$
which shows that $\mu$ is invariant.

\item \emph{Step 3.} We now show that $T$ is $\mu$-ergodic. Suppose, for the sake of contradiction, that there exists a set $A \in \mathcal{F}$, invariant under $T$, such that $0 < \mu(A) < 1$. This also implies $0 < \mu(\Omega \setminus A) < 1$. Two cases arise :
\begin{itemize}
    \item[---] If $x_0 \in A$, since $T^{-1}(A) = A$, the iterates  $T^n(x_0)$ remain in $A$, so there does not exist any $n$ such that $ T^n(x_0) \in \Omega \setminus A$, which is a contradiction because the sequence is dense in $\Omega$.
    \item[---] If $x_0 \notin A$, since $T^{-1}(A) = A$, we have $T^{-1}(\Omega \setminus A) = \Omega \setminus A$, so the iterates  $T^n(x_0)$ remain in $\Omega \setminus A$, which is a contradiction by the same reasoning.
\end{itemize}
Thus, $\mu(A) = 0$ or $\mu(A) = 1$. Therefore, $T$ is $\mu$-ergodic.

Consequently, according to Birkhoff's theorem \cite{BI}, for any function $f \in L^1(\Omega, \mu)$, the time averages
$$
\frac{1}{N} \sum_{n=0}^{N-1} f(T^n(x_0))
$$
converge as $N \to \infty$ towards the spatial average $E_{\mu}(f) = \int_\Omega f \, d\mu$ for almost every $x_0 \in \Omega$ with respect to $\mu$.
In particular, for $\mu$-almost every point $x_0 \in \Omega$,
$$
\lim_{N \to \infty} \frac{1}{N} \sum_{n=0}^{N-1} \mathrm{1}_A(T^n(x_0)) = \mu(A),
$$
for all $A \in \mathcal{F}$. This implies, in particular, that if $U \subset \Omega$ is an open set, there exists $n \in \mathbb{N}$ such that $T^n(x_0) \in U$. Thus, the dynamics are dense $\mu$-almost everywhere in $\Omega$. As a result, the dynamical system $(\Omega, T)$ is chaotic with respect to $\mu$, and therefore there exists a chaotic sequence $(u_n)$ in $\Omega$ such that every point in $\Omega$ is the limit of a convergent subsequence of the sequence $(u_n)$, since it is dense in $\Omega$.
It follows that every point (which is an elementary event) $\omega$ in the probability space $(\Omega, \mathcal{F}, \mu)$ is the limit of a subsequence of the chaotic sequence $(u_n)$.

\item \emph{Step 4.} To complete the proof, we now show that any random phenomenon on $\Omega$ can be approximated by a chaotic phenomenon on $\Omega$. Consider a random experiment \( \mathcal{A} \) whose universe is \( \Omega \) and whose sigma-algebra is \( \mathcal{G} \). 

Since the $\sigma$-algebra $\mathcal{G}$ on $\Omega$ is countably generated, there exists a countable family of sets $\{A_n\}_{n \geq 1} \subset \mathcal{G}$ such that
$$
\mathcal{G} = \sigma\left(\{A_n, \ n \geq 1\}\right).
$$
Let $X_0$ be the function defined by
\begin{align*}
    X_0 \ :\ & \Omega \to \{0,1\}^{\mathbb{N}}, \\
    & \omega \mapsto \left( \mathbf{1}_{A_1}(\omega), \mathbf{1}_{A_2}(\omega), \mathbf{1}_{A_3}(\omega), \dots \right).
\end{align*}
For each $n$, the function $\omega \mapsto \mathbf{1}_{A_n}(\omega)$ is measurable and then $X_0$ is a random variable from $(\Omega, \mathcal{B}(\Omega))$ into $(\{0,1\}^{\mathbb{N}}, \mathcal{B}(\{0,1\}^{\mathbb{N}}))$. Moreover, we have $\sigma(X_0) = \mathcal{G}$. Indeed, by definition, the $\sigma$-algebra $\sigma(X_0)$ is the smallest $\sigma$-algebra on $\Omega$ for which $X_0$ is measurable, that is 
$$
\sigma(X_0) = \{X_0^{-1}(B) : B \in \mathcal{B}(\{0,1\}^{\mathbb{N}})\}.
$$
On the other hand, for each $n$, we have
$$
\{\omega \in \Omega : \mathbf{1}_{A_n}(\omega) = 1\} = A_n,
$$
which is obtained as the preimage under $X_0$ of a cylinder set in $\{0,1\}^{\mathbb{N}}$.

Since the cylinders (sets fixing one or finitely many coordinates) generate the Borel $\sigma$-algebra $\mathcal{B}(\{0,1\}^{\mathbb{N}})$, we deduce that all these $A_n$ are in $\sigma(X_0)$. 

As $\sigma(X_0)$ is a $\sigma$-algebra containing all the $A_n$, and since $\mathcal{G} = \sigma(\{A_n, \ n\geq1\})$ is by definition the smallest $\sigma$-algebra containing these sets, it follows that 
$$
\mathcal{G} \subset \sigma(X_0).
$$ 
Conversely, by the very construction of $X_0$, all the information in $\sigma(X_0)$ is expressed in terms of the $A_n$. Thus, we also obtain 
$$
\sigma(X_0) \subset \mathcal{G}.
$$

Otherwise, since $\Omega$ is a compact metric space and similarly, $\{0,1\}^{\mathbb{N}}$ is a compact metrizable space, hence there exists a Borel isomorphism $\varphi : \{0,1\}^{\mathbb{N}} \to \Omega$, that is, a measurable bijection whose inverse is also measurable.

Define  $$
X = \varphi \circ X_0,
$$
so that we obtain a random variable
$$
X : (\Omega, \mathcal{B}(\Omega)) \to (\Omega, \mathcal{G}),
$$
which is measurable since both $X_0$ and $\varphi$ are measurable. Moreover, the $\sigma$-algebra generated by $X$ is given by
$$
\sigma(X) = \sigma(\varphi \circ X_0) = \sigma(X_0) = \mathcal{G},
$$
because the composition with a Borel isomorphism does not change the generated $\sigma$-algebra.

We can now define a probability measure $P$ that assigns to each event $A \in \mathcal{G}$ its probability $P(A)$. The probability $P$ is defined as the image measure $\mu$ by $X$
$$
P(A) = \mu(X^{-1}(A)), \quad \text{for all } A \in \mathcal{G}.
$$
Which completes the proof.

\end{itemize}

\end{proof}

\begin{rem}
    \rm{The invariant measure constructed in the precedent proof can be interpreted as  the natural probability law of the elements of the universe $\Omega$, associated with the random experiment where an element of $\Omega$ is randomly selected.}
\end{rem}

Now we can get the proof of the Corollary 2.3 :

\begin{proof} 
Consider a nonzero Radon measure $\nu$ on $\Omega$ such that the trajectories of $T$ are dense $\nu$-almost everywhere in $\Omega$. Let $x_0 \in \Omega$ such that the sequence $(T^n x_0)$ is dense in $\Omega$, then we can extract a convergent subsequence of $(T^n x_0)$. We define the unpredictable points $\omega \in \Omega$ as follows:
\begin{center}
    "$\omega$ is unpredictable if and only if it is the limit of a convergent subsequence of the sequence $(T^n x_0)$".
\end{center} 

By following the same steps 2 and 3 from the proof of Theorem 2.2., we show that there exists an invariant measure $\mu$ associated with the dynamical system $(\Omega, T)$ such that $T$ is $\mu$-ergodic. As a result, the space $(\Omega, \mathcal{F}, \mu)$ is probabilized and every unpredictable point of $\Omega$ is the limit of a subsequence of the chaotic sequence $(u_n) = (T^n x_0)$. Since the sequence $(u_n)$ is dense in $\Omega$, all points in $\Omega$ are unpredictable. Hence, every point of $\Omega$ is the limit of a subsequence of the chaotic sequence $(u_n) = (T^n x_0)$. This completes the proof.

\end{proof}

\begin{cor}
Let $X$ be a random vector defined on a probability space with values in $\mathbb{N}^d$, $\mathbb{Z}^d$, or $\mathbb{R}^d$. Then, any random phenomenon that can be modeled by $X$ can be approximated by a chaotic phenomenon.
\end{cor}

\begin{proof} 

Let $X$ be a random vector with values in $\mathbb{N}^d$, $\mathbb{Z}^d$, or $\mathbb{R}^d$ with  cumulative distribution function $F_X$. 

\begin{itemize}
    \item If $d = 1$, then we have the following equality in law $X = F_X^{-1}(U)$, where $U$ is a random variable uniformly distributed on $[0,1]$ and $F_X^{-1}$ is the generalized inverse of $F_X$ (or quantile function) defined by
    $$
    F^{-1}(u) = \inf\{y \in \mathbb{R} \mid F(y) \geq u\}, \ \forall u\in[0,1].
    $$
    The result follows from the previous theorem by applying it to $\Omega = [0, 1]$.
    
    \item If $d \geq 2$, according to the Sklar's theorem, the multivariate cumulative distribution function $F_X$ of the random vector $X = (X_1, \dots, X_d)$ can be written: for all $(x_1,x_2,...,x_d)\in\mathbb{R}^d$,
    $$
    F_X(x_1, \dots, x_d) = C\big(F_{X_1}(x_1), \dots, F_{X_d}(x_d)\big),
    $$
    where $F_{X_i}$ are the marginal cumulative distribution functions of $X_i$ and $C$ is a copula, capturing the dependence structure between the $X_i$, that is, a multivariate cumulative distribution function of a random vector $U=(U_1,...,U_d)$ uniformly distributed on $[0,1]^d$ (i.e., whose margins are uniform on $[0,1]$). In this case, we have the following equality in law $X = (F_{X_1}^{-1}(U_1),..., F_{X_d}^{-1}(U_d))$. Now the result follows directly from the previous theorem by applying it to $\Omega = [0, 1]^d$.
\end{itemize}

\end{proof}

\begin{rem}
    \rm{   A dynamical system $ (\Omega, T) $ is chaotic in the sense of Devaney \cite{De} if: $ T $ is transitive;   $ T $ is sensitive to initial conditions; and the periodic points $ x_0 \in \Omega $ of $ T $ (i.e., there exists $ n_0 $ such that $ T^{n_0}(x_0) = x_0 $) are dense in $ \Omega $. The density of the set of periodic points ensures a structural richness in the system. This condition, combined with topological transitivity, shows that chaos is not simply a "total disorganization," but that there is some underlying coherence. It is this combination of local regularity and global complexity that characterizes chaos in the sense of Devaney.}
\end{rem}

The following result shows that the set of periodic points is negligible with respect to the invariant measure, particularly when $ \Omega $ is assumed to be compact. Thus, the chaos defined in Definition 2.1 constitutes a generalization of that proposed by Devaney \cite{De}.

\begin{prop}
    Under the conditions of Theorem 2.2, the periodic points of $T$ form a negligible set with respect to the invariant measure.
\end{prop}
\begin{proof}Let $\mu$ be the invariant measure constructed on $\Omega$ in the proof of the previous theorem. The transformation $ T $ is $ \mu $-ergodic by construction. Let $ P $ be the set of periodic points of $ T $.
We have $$P = \bigcup_{n=1}^{\infty} P_n,$$ where $P_n =\left\{ x \in \Omega : T^n(x) = x \right\}$ denotes the periodic points of period $n$. Since
$$
\mu(P) \leq \sum_{n=1}^{\infty} \mu(P_n), 
$$ it suffices to show that $\mu(P_n)=0$ for all $n$.
For each $ n $, $ P_n $ is discrete (it consists of fixed points or discrete cycles), because each periodic point corresponds to a solution of an iteration equation that takes a finite number of values in a compact space. Since $ T $ is continuous, the application $ T^n $ is also continuous. The set $ P_n $ is therefore the set of fixed points of the continuous application $ T^n $. According to the fixed-point theorem, $ P_n $ is closed in $ \Omega $, and consequently, $ P_n $ is finite in $ \Omega $.
On the other hand, let $ x \in T^{-1}(P_n) $ so $ T(x) \in P_n $, that is, $ T(x) $ is a periodic point of $ T $, meaning $ T^{n+1}(x) = T(x) $. This shows that $ x $ is also a periodic point of $ T $, and therefore $ x \in P_n $. Thus, $ T^{-1}(P_n) \subseteq P_n $. This shows that $ P_n $ is invariant under $ T $.
Since $T$ is $\mu$-ergodic, then $\mu(P_n)=0$ or $\mu(P_n)=1$. If $\mu(P_n)=1$, the measure $\mu$ is supported by the periodic points of period $n$, but the trajectories of these points cannot satisfy the density criterion of their trajectories in $\Omega$ since they form finite cycles.
which is contradictory to the fact that $(\Omega,T)$ is chaotic. As a result, $\mu(P_n) = 0$,  which completes the proof.\end{proof}

\section{Illustrative examples}

The proposition 2.13  reveals that the topological density of periodic points constitutes an excessively strong condition. Although the set of periodic points is dense in $\Omega$, it has measure zero under the invariant measure $\mu$. ccording to $\mu$ is precisely zero, underscoring the exceptional nature of these points from a measure-theoretic perspective. To illustrate this finding, we will examine several well-established examples from the literature.

\subsection{The logistic model}

The logistic model is both a simple and powerful example of a deterministic system capable of producing chaotic behavior, while also serving as a generator of pseudo-random sequences. This phenomenon can be understood through the fundamental properties of dynamical systems, and its rigorous justification relies on Corollary 2.11, which demonstrates that each elementary event of $\Omega$ is the limit of a subsequence of the chaotic logistic sequence $(x_n)$.\\
Let $\Omega=[0,1]$.
The transformation associated with the logistic sequence is the following quadratic function
$$
T_L(x) = \lambda x (1 - x),
$$
where $ \lambda $ is a real parameter in the interval $[0,4]$ and $ x \in \Omega $. This transformation acts on a point $ x_n $ to produce the next one in the sequence
$$
x_{n+1} = T_L(x_n) = \lambda x_n (1 - x_n).
$$
This iteration generates an orbit $ (x_n) $, which depends on the initial condition $ x_0 $ and the properties of $ T_L $. For $ \lambda= 4 $, the transformation $ T_L(x) = 4x(1 - x) $ is well-studied on $ \Omega $ (see \cite{Fe}, \cite{Ma}, and \cite{Sc}) and associated with the invariant measure $\mu_L$ which is absolutely continuous  with respect to the Lebesgue measure, whose invariant density is given by
$$
\rho_L(x) = \frac{1}{\pi \sqrt{x(1 - x)}}\mathrm{1}_{(0,1)}(x), \ \forall  x\in \mathbb{R}.
$$
This density is related to the distribution of points generated by the logistic sequence $(x_n)$, which covers the space $\Omega$ in a chaotic manner, but not uniformly, with a greater concentration of points near the edges $0$ and $1$ (see figure below).

\begin{center}
\includegraphics[scale=.7]{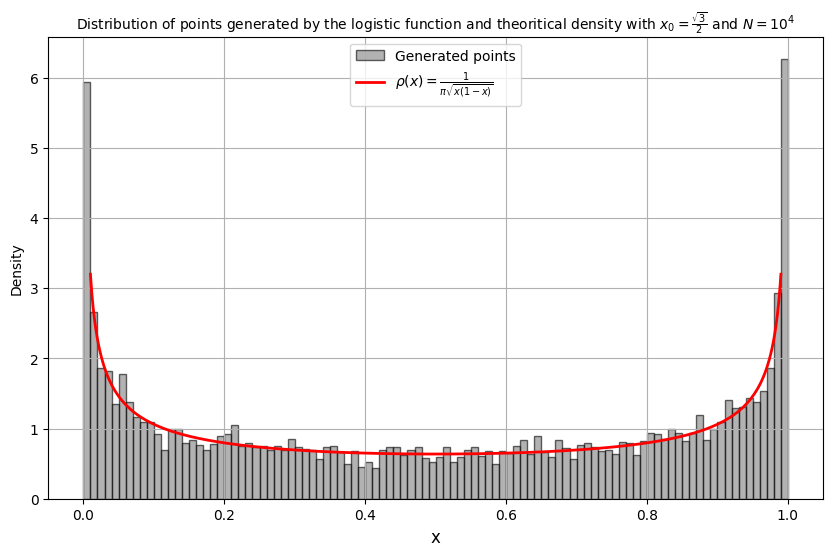}    
\end{center}

The system is chaotic on $ \Omega $ since the sequence $(x_n)$ is dense in $\Omega$ for $\mu_L$-almost every initial point $x_0$ in $\Omega$. The initial point $ x_0 $ must be chosen in the interval $ (0, 1) $ to observe chaotic behavior. This excludes the points $ x_0 = 0 $ and $ x_0 = 1 $ as well as certain specific points that may belong to   periodic orbits (such as fixed points or limit cycles when $ \lambda $ is below the chaotic threshold). However, these points are rare and form a set of (Lebesgue) measure  zero.\\
The sequence $ (x_n) $ is dense in $ \Omega $ for $ \lambda = 4 $, which means that the points generated by the sequence get arbitrarily close to any point in $ \Omega $, including points near the edges of $ \Omega $ where the density $ \rho_L $ is particularly high. Taking into account the invariant measure, each point of $ \Omega $ can be interpreted as an elementary event, and the logistic sequence produces points distributed according to this density. Thus, each point of $ \Omega $ can be approximated as the limit of a convergent subsequence of the chaotic sequence $ (x_n) $.\\
To transform these values into a uniform distribution on $[0,1]$, the cumulative distribution function is used
$$
F(x) = \int_{-\infty}^x \rho_L(t) \, dt, \quad x\in\mathbb{R}.
$$
By applying the following transformation to each value $x_n$ generated by the logistic sequence $u_n = F(x_n),$
we obtain a sequence $(u_n)$ that is uniformly distributed on $[0,1]$.
For a variable $X$ uniformly distributed on $[0,1]$, one can generate a new random variable $Y$ following a distribution with a cumulative distribution function $F_Y(x)$ using the following generalized inverse method
$$
Y = F_Y^{-1}(X),
$$
where $F_Y^{-1}(u)$ is the generalized inverse of $F_Y(x)$, defined as follows
$$
F_Y^{-1}(u) = \inf \{x \in \mathbb{R} : F_Y(x) \geq u\}, \quad u \in [0,1].
$$
To generate values according to the law $Y$, we follow the following steps:
\begin{enumerate}
    \item[\emph{(i)}] Generation of a uniform sequence $(u_n)$ from a sequence of pseudo-random numbers obtained from a transformation of the logistic sequence.
  \item[\emph{(ii)}] Application of the inverse transform: For each $u_n \in [0,1]$, calculate $y_n = F_Y^{-1}(u_n)$.
    \item[\emph{(iii)}]  Use of the sequence $(y_n)$: The $y_n$ follow the law of the random variable $Y$ defined by the cumulative distribution function $F_Y$.
\end{enumerate}
 
Consequently, the machine generation of all standard distributions of random variables, random vectors, as well as stochastic process distributions (by the logistic sequence, for example), relies on a rigorous mathematical result that guarantees their validity.

\subsection{The Gauss-Kuzmin-Wirsing model}

 Let's take  again the Gauss-Kuzmin-Wirsing transformation $T_G$ defined in the example 2.7 on $\Omega=[0,1]$. The dynamical system $(\Omega,T_G)$ is chaotic because the Gauss transformation is ergodic with respect to its invariant measure $\mu_G$ using Lemma 2.1. This measure is absolutely continuous with respect to the Lebesgue measure on $\Omega$ 
$$
\mu_G(A) = \int_A \frac{1}{\ln 2} \cdot \frac{1}{1 + x} \, dx, 
$$
for all  $A \in \mathcal{F}$. A graphical illustration of the points generated by this transformation, along with the associated density curve $\rho_G$, is presented in the following figure:
\begin{center}
\includegraphics[scale=.7]{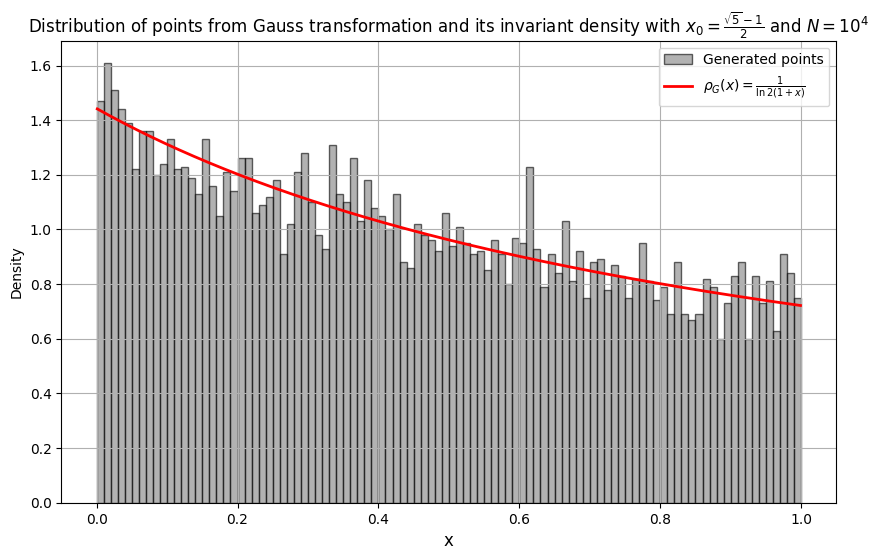}    
\end{center}

To determine the periodic points of the Gauss transformation $ T_G(x) = \frac{1}{x} - \lfloor \frac{1}{x} \rfloor $, it is necessary to solve the equation $ T_G^n(x) = x $, where $ T_G^n $ is the $n$-th iteration of $ T_G $. In particular, the fixed points (period $ n = 1 $) satisfy $ T_G(x) = x $. This gives the equation 
$$
\frac{1}{x} - \left\lfloor \frac{1}{x} \right\rfloor = x.
$$
The point 0 is a fixed point of the Gauss map, thus a trivial case of a periodic point with a period of 1. Let $ k = \lfloor \frac{1}{x} \rfloor $ be an integer, then 
$$ x^2 + kx - 1 = 0.$$

The solutions to this equation are given by 
\[
x = \frac{-k \pm \sqrt{k^2 + 4}}{2}.
\]
However, $x \in (0, 1)$, so we only keep the positive solution
\[
x = \frac{-k + \sqrt{k^2 + 4}}{2}.
\]
For each integer $k \geq 1$, this gives a fixed point. For  $k = 1$, $x = \frac{-1 + \sqrt{5}}{2}$ (the reciprocal of the golden ratio).\\
The set of periodic points of $T_G$ is dense in $\Omega$, and these points correspond to the real numbers in $\Omega$ that have a purely periodic continued fraction expansion. This includes all irrational quadratic numbers, that is, the solutions of quadratic equations with integer coefficients and positive discriminant. The set of periodic points of the Gauss transformation is countable and, consequently, has measure zero with respect to the Lebesgue measure.

\subsection{The Hénon model}

The Hénon model was initially used to model astrophysical phenomena, particularly the dynamics of binary systems and galaxies \cite{CA}. The Hénon map is defined on $\mathbb{R}^2$ by
\[
T_H(x, y) = \left(1 - a x^2 +  y ,b x \right),
\]
where $a$  and $b$ are real parameters of the system. The Hénon model is an emblematic example of a chaotic dynamical system within the framework of discrete systems, particularly suited for the study of chaos in 2-dimensional space. This system is generally studied on the space $\mathbb{R}^2$, and its equation is given by
\[
\begin{pmatrix} x_{n+1} \\ y_{n+1} \end{pmatrix} = 
\begin{pmatrix} 1 - a x_n^2 + y_n \\ b x_n \end{pmatrix},
\]
where $(x_n, y_n)$ are the coordinates at time $n$.

The parameters $a$ and $b$ allow the generation of strange attractors in the phase space, with properties such as sensitivity to initial conditions and a fractal structure. To achieve fully chaotic behavior in the Hénon system, the most commonly used parameter values are 
\[
a = 1.4 \quad \text{and} \quad b = 0.3
\]
\cite{He}. Thus, for an appropriate choice of the initial point $(x_0, y_0)$, the iterations of the map $T_H$ densely cover a compact subset $\Omega$ of the phase space. This set $\Omega$, called the Hénon attractor, is a fractal object exhibiting a complex structure and remarkable topological properties, such as its fractal dimension, its connection to chaotic behaviors, and its non-periodic nature. Numerical estimates have been used to estimate the Hausdorff dimension of $\Omega$, which is approximately 1.26, providing an accurate measure of the geometric complexity of this attractor \cite{Ru}.\\

In the particular case of the Hénon map with $a = 1.4$ and $b = 0.3$, the exact expression of the invariant density is difficult to obtain directly due to the complexity of the dynamic system, particularly because of the presence of the fractal structure of $\Omega$. However, this density can be approximated using numerical methods such as the empirical distribution method, which involves observing the frequency of points visited over the iterations, or the Monte Carlo simulations method, which generates many random trajectories to estimate the system's probability distribution. These numerical methods provide efficient approximations of the invariant density, which is often localized in a specific region of phase space, while exhibiting fractal and chaotic characteristics, particularly for the chosen parameters. The following figure illustrates the distribution of points generated by the Hénon map and the associated invariant measure density  with $a = 1.4$, $b = 0.3$, $(x_0,y_0)=(0.1,0.1)$ and $N=10^4$.\\

\hspace{-1cm}\includegraphics[scale=.7]{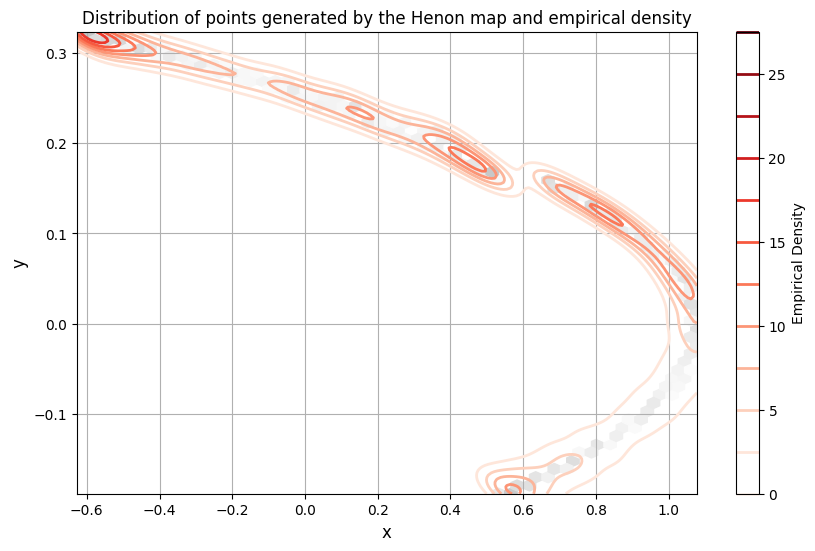}

\section{Conclusion}

In conclusion, this article presents an innovative approach aimed at demonstrating that a random phenomenon can be approximated by a chaotic phenomenon on compact metric spaces, where we have shown that any elementary event in a probabilistic space can be approximated by a convergent subsequence of a chaotic phenomenon. This goes beyond classical statistical tests used to evaluate pseudo-random number generators.\\ 
A natural and promising direction for future research lies in extending these findings to broader settings, such as Polish spaces, locally compact metric spaces, and other more general topological frameworks. Such generalizations could not only deepen the theoretical foundations of this work but also expand its practical applicability, paving the way for novel advancements in the field.

\section*{Conflicts of interst}
The authors declares no conflict of interest.

%%%%%%%%%%%%%%%%%%%%%%%%%%%%%%%%%%%%%%%%

\end{document}